\documentclass[12pt]{article}
\usepackage{amsfonts,amsmath,amscd,amssymb,amsthm, graphics,color}

\newfont{\eufm}{eufm10 scaled\magstep1}

\newcommand{\cA}{\mathcal{A}}
\newcommand{\cC}{\mathcal{C}}

\newcommand{\cD}{\mathcal{D}}

\newcommand{\cV}{\mathcal{V}}

\newcommand{\cP}{\mathcal{P}}
\newcommand{\cF}{\mathcal{F}}

\newcommand{\cL}{\mathcal{L}}
\newcommand{\cG}{\mathcal{G}}

\newcommand{\cR}{\mathcal{R}}

\newcommand{\bbN}{\mathbb{N}}

\newcommand{\bbC}{\mathbb{C}}

\newcommand{\bbQ}{\mathbb{Q}}

\newcommand{\bbF}{\mathbb{F}}
\newcommand{\bbE}{\mathbb{E}}

\newcommand{\bbK}{\mathbb{K}}

\newcommand{\bbD}{\mathbb{D}}

\newtheorem{thm}{Theorem}[section]
\newtheorem{lem}[thm]{Lemma}
\newtheorem{cor}[thm]{Corollary}
\newtheorem{prop}[thm]{Proposition}
\newtheorem{defi}[thm]{Definition}
\newtheorem{rem}[thm]{Remark}

\newtheorem{ex}[thm]{Example}
\newtheorem{alg}[thm]{Algorithm}

\def\para{\vspace{2mm}}
\def\id{{\rm ID}}
\def\dres{\partial{\rm Res}}
\def\dcres{\partial{\rm CRes}}
\def\ps{{\rm PS}}
\def\zero{{\rm Zero}}
\def\rank{{\rm rank}}
\def\ord{{\rm ord}}
\def\lead{{\rm lead}}
\def\prem{{\rm prem}}
\def\min{{\rm min}}
\def\gcrd{{\rm gcrd}}

\def\im{{\rm IM}}

\begin{document}
\title{Linear Complete Differential Resultants and the Implicitization of Linear DPPEs\thanks{Both authors
supported by the Spanish ``Ministerio de Educaci\'on y Ciencia"
under the Project MTM2005-08690-C02-01}}

\author{Sonia L. Rueda   \\
Dpto de Matem\' atica Aplicada, E.T.S. Arquitectura\\
Universidad Polit\' ecnica de Madrid\\
Avda. Juan de Herrera 4, 28040-Madrid, Spain\\
{sonialuisa.rueda@upm.es}\and J. Rafael Sendra \\ Dpto de
Matem\'aticas. Universidad de Alcal\'a. \\ E-28871 Madrid, Spain. \\
{rafael.sendra@uah.es}}

\maketitle

\begin{abstract}
The linear complete differential resultant of a finite set of
linear ordinary differential polynomials is defined. We study the
computation by linear complete differential resultants of the
implicit equation of a system of $n$ linear differential
polynomial parametric equations in $n-1$ differential parameters.
We give necessary conditions to ensure properness of the system of
differential polynomial parametric equations.
\end{abstract}

\section{Introduction}

The implicitization problem of unirational algebraic varieties has
been widely studied and the results on the computation of the
implicit equation of a system of algebraic rational parametric
equations by algebraic resultants are well known \cite{Cox}, \cite{Cox-2}. This
work was motivated by the paper of X.S. Gao \cite{Gao}, where
computational issues, related with the implicitization problem of
differential rational parametric equations, are treated by
characteristic set methods, making use of the differential algebra
theory developed by Ritt \cite{Ritt} and Kolchin \cite{Kol}. The
paper by X.S. Gao \cite{Gao}, establishes the basis ground for
the generalization to the differential case of the results in
algebraic geometry on implicit and parametric representations of
unirational varieties, conversion algorithms, etc (see, for
instance, \cite{Cox}, \cite{Cox-2}, \cite{sendra-libro}).

We explore the first steps of the generalization to the differential case of the
results in algebraic geometry on implicit representations of
unirational varieties. To be more precise, we are interested in
finding a differential resultant that would solve the
differential implicitization problem. The implicitization problem of
differential rational parametric equations is a differential
elimination problem. There is a wide range of applications of the
differential elimination method to computer algebra and applied
mathematics. For a survey on differential elimination techniques
and their application to biological modeling we refer to
\cite{Bo}.

We defined the implicit equation of a system of $n$ differential
rational parametric equations in $n-1$ differential parameters in
\cite{RS}. In this paper, we study the computation by
differential resultants of the implicit equation of a system $\cP
(X,U)$ of $n$ linear differential polynomial parametric equations
(linear DPPEs) in $n-1$ differential parameters $u_1,\ldots
,u_{n-1}$ (we give a precise statement of the problem in Section
~\ref{sec-Basic notions and notation}),
\begin{equation*}
\cP(X,U) = \left\{\begin{array}{ccc}x_1 &= & P_1 (U)\\ & \vdots  & \\
x_n&= & P_n (U)\end{array}\right..
\end{equation*}

The differential resultant of a set of ordinary differential
polynomials  was introduced and studied by G. Carra'Ferro in
\cite{CF97} for two differential polynomials and in
\cite{CFproc} for a finite number $n$ of differential polynomials
in $n-1$ differential variables. The generalized differential
resultant was also defined by G. Carra'Ferro in \cite{CF05} for
$n+s$ differential polynomials in $n-1$ differential variables,
$s\geq 0$. Previous definitions of the differential resultant for
differential operators are due to L.M. Berkovich-V.G. Tsirulik
\cite{BT} and M. Chardin \cite{Ch}.

Let us consider the linear ordinary differential polynomials
$F_i(X,U)=x_i-P_i (U)$ of order $o_i$, $i=1,\ldots ,n$. As we will
explain in Section ~\ref{subsec-Differential resultants of
F_1,...,F_n} the differential resultant $\dres (F_1,\ldots ,F_n)$
is the determinant of the $L\times L$ matrix $M(L)$, where $L=\sum_{i=1}^{n} ((\sum_{k=1}^n
o_k)-o_i+1)$. In Section
~\ref{sec-Implicitization of linear DPPEs I}, we prove that if
$\dres (F_1,\ldots ,F_n)\neq 0$ then the implicit equation of the
system $\cP (X,U)$ is given by the differential polynomial $\dres
(F_1,\ldots ,F_n)(X)$. Then we analyze some of the reasons for
$\dres (F_1,\ldots ,F_n)= 0$.

The differential resultant $\dres (F_1,\ldots ,F_n)$ is the
Macaulay's algebraic resultant of a differential polynomial set
$\ps$ with $L$ elements. One reason for $\dres (F_1,\ldots ,F_n)= 0$ is the following:
if the polynomials in $\ps$ are not
complete in all its variables then the resultant will be zero. If
for some $j\in\{1,\ldots ,n-1\}$ the differential variable $u_j$
has order less than $o_i$ in $F_i$ for all $i\in \{1,\ldots ,n\}$
then the matrix $M(L)$ has one or more columns of zeros. Let
$\gamma$ be the number of zero columns in $M(L)$ due to this
reason. Let us see an illustrating example.

\begin{ex}\label{ex-intro}
Let us consider the following system of linear DPPEs
\begin{equation*}
\left\{\begin{array}{ccl}
x_1&=& u_1+u_2+u_{21}\\
x_2&=& tu_{11}+u_{22}\\
x_3&=& u_1+u_{21}
\end{array}\right.
\end{equation*}
where $u_{jk}=\partial^k u_j/\partial t^k$, $j=1,2$ and $k\in\bbN$. The differential resultant of
$F_1(X,U)=x_1-u_1-u_2-u_{21}, F_2(X,U)=x_2-tu_{11}-u_{22}, F_3(X,U)=x_3-u_1-u_{21}$ is zero because
the order of $u_1$ in every polynomial is less than its order,
$\gamma=1$. The implicit equation of this systems is
$(t-1)x_{12}-tx_{31}-(t-1)x_{32}+x_2$, $x_{ik}=\partial^k x_i/\partial t^k$, $i=1,2,3$, $k\in\bbN$.
\end{ex}

We define in Section ~\ref{sec-Linear gamma-Differential
Resultants} the linear complete differential resultant of a finite
set of linear ordinary differential polynomials, generalizing
Carra'Ferro's differential resultant in the linear case. We prove
the next theorem in Section ~\ref{sec-Linear gamma-Differential
Resultants}.

\vspace{0.5cm}{\bf \noindent Theorem.}
\begin{em}
Given a system $\cP (X,U)$ of differential polynomial parametric
equations. If the linear complete differential resultant
$\dcres (F_1,\ldots ,F_n)\neq 0$ then $\dcres
(F_1,\ldots ,F_n)(X)=0$ is the implicit equation of $\cP (X,U)$.
\end{em}
\vspace{0.5cm}

In Section ~\ref{sec-Some results on properness}, we study whether
the system $\cP (X,U)$ of linear DPPEs is proper, problem closely
related to the existence of inversion maps. If the system $\cP
(X,U)$ is not proper we prove that $\dcres (F_1,\ldots
,F_n)=0$. The homogeneous part of the linear differential
polynomials $F_1,\ldots ,F_n$ can be written in terms of
differential operators. We obtain necessary conditions on these
differential operators so that $\dcres (F_1,\ldots
,F_n)\neq 0$.

Computations throughout this paper were carried out with our Maple implementation
of functions to compute the differential resultant by Carra'Ferro in \cite{CFproc}
and the linear complete differential resultant defined in Section ~\ref{sec-Linear gamma-Differential
Resultants} of this paper. Our implementation is available at \cite{R}.

The paper is organized as follows. In Section ~\ref{sec-Basic
notions and notation} we introduce the main notions and notation.
Next we review the definition of the differential resultant
defined by Carra'Ferro in Section ~\ref{sec-Differential
Resultants} and we define the differential homogeneous resultant.
In Section ~\ref{sec-Implicitization of linear DPPEs I} we explain
the computation of the implicit equation by Carra'Ferro's
differential resultant. We define the linear complete differential
resultant in Section ~\ref{sec-Linear gamma-Differential
Resultants}. In Section ~\ref{sec-Implicitization of linear DPPEs}
we give our main results on the implicitization of linear DPPEs by
linear complete differential resultants. Our results on properness
appear in Section ~\ref{sec-Some results on properness}. We finish
the paper discussing the special cases $n=2$ and $n=3$ in Section
~\ref{sec-Treatment of special cases}.

\section{Basic Notions and Notation}\label{sec-Basic notions and
notation}

In this section, we introduce the basic notions related to the
problem we deal with, as well as the notation and terminology used
throughout the paper. For further concepts and results on
differential algebra we refer to  \cite{Kol} and \cite{Ritt}.

\para

Let $\bbK$ be an ordinary differential field with derivation
$\partial$, ( e.g. $\bbQ(t)$, $\partial=\frac{\partial}{\partial
t}$). Let $X=\{x_1,\ldots ,x_n\}$ and $U=\{u_1,\ldots ,u_{n-1}\}$
be  sets of differential indeterminates over $\bbK$. Let $\bbE$ be
a universal extension field of $\bbK$ containing the set of
indeterminates $U$. Let $\bbN_0=\{0,1,2,\ldots ,n,\ldots\}$. For
$k\in\bbN_0$ we denote by $x_{ik}$ the $k$-th derivative of $x_i$.
Given a set $Y$ of differential indeterminates over $\bbK$ we
denote by $\{Y\}$ the set of derivatives of the elements of $Y$,
$\{Y\}=\{\partial^k y\mid y\in Y,\; k\in \bbN_0\}$, and  by $\bbK
\{X\}$ the ring of differential polynomials in the differential
indeterminates $x_1,\ldots ,x_n$, that is
\begin{displaymath}
\bbK\{X\}=\bbK[x_{ik}\mid i=1,\ldots ,n,\; k\in \bbN_0 ].
\end{displaymath}
Analogously for $\bbK \{U\}$.

\para

A {\sf system of differential rational parametric equations} (system
of DRPEs) is a system of the form
\begin{equation}\label{DRPEs}
{\cal Q}(X,U)=\left\{\begin{array}{ccc}x_1 &=& \frac{P_1 (U)}{Q_1(U)}\\ &\vdots & \\
x_n &=& \frac{P_n (U)}{Q_n(U)}\end{array}\right.
\end{equation}
where $P_1,\ldots ,P_n, Q_1,\ldots ,Q_n\in \bbK \{U\}$,  $Q_i$ are
non-zero, and  not all $P_i, Q_i\in \bbK$, $i=1,\ldots ,n$. We call
the indeterminates $U$ a set of parameters of ${\cal Q}(X,U)$; they
are not necessarily independent.  When all differential polynomials
$P_i,Q_i$ are of degree at most 1, we say that ~\eqref{DRPEs} is a
{\sf linear system.} Moreover, if all $Q_i\in \bbK$, we say that
~\eqref{DRPEs} is a {\sf system of differential polynomial
parametric equations} (system of DPPEs).
 Associated with the
system ~\eqref{DRPEs} we consider the differential ideal (see
\cite{Gao}, Section 3)
\begin{equation*}
\id =\{f\in \bbK \{X\}\mid f(P_1 (U)/Q_1(U),\ldots ,P_n
(U)/Q_n(U))=0\},
\end{equation*}
and we call it the {\sf implicit ideal}
 of the system ~\eqref{DRPEs}.
By \cite{Gao}, Lemma 3.1, the implicit ideal $\id$ is a
differential prime ideal. Moreover, given a characteristic set
$\cC$ of $\id$ then $n-\mid \cC \mid$ is the (differential)
dimension of $\id$, by abuse of notation, we will also speak about
the dimension of a DRPEs system meaning the dimension of its
implicit ideal. The parameters $U$ are independent if  $\dim({\cal
Q}(X,U))=|U|$  (see \cite{Gao}, Section 5).

\para

The differential variety defined by
\[ \zero(\id)=\{\eta\in \Bbb E^n\mid \forall f\in\id, f(\eta)=0\} \]
is called the {\sf implicit variety} of ${\cal Q}(X,U)$. If the
parameters of ${\cal Q}(X,U)$ are not independent and $\bbK=\bbQ
(t)$, there exists a set of new DRPEs with the same implicit
variety as ${\cal Q}(X,U)$ but with independent parameters, (see
\cite{Gao}, Theorem 5.1).

\para

If $\dim(\id)=n-1$, then $\cC=\{A(X)\}$ for some irreducible
differential polynomial $A\in\bbK\{X\}$. The polynomial $A$  is
called a {\sf characteristic polynomial} of $\id$. Furthermore,  if
$B$ is another characteristic polynomial of $\id$ then $A=bB$ with
$b\in \bbK$.

\para

We introduced the notion of implicit equation in \cite{RS} and we
include it here for completion.

\begin{defi}
 The {\sf implicit equation} of a $(n-1)$-dimensional
 system of DRPEs, in $n$ differential indeterminates $X=\{x_1,\ldots,x_n\}$,
 is defined as the equation
 $A(X)=0$, where $A$ is any characteristic polynomial of the
 implicit ideal $\id$ of the system.
\end{defi}

In the sequel, we consider the linear system of DPPEs
\begin{equation}\label{DPPEs}
\cP(X,U) = \left\{\begin{array}{ccc}x_1 &= & P_1 (U)\\ & \vdots  & \\
x_n&= & P_n (U)\end{array}\right.
\end{equation}
So,  $P_1,\ldots ,P_n\in \bbK\{U\}$ with degree at most $1$, and
not all $P_i\in \bbK$, $i=1,\ldots ,n$.

Let $\bbK [\partial]$ be the ring  of differential operators with
coefficients in $\bbK$. For $i=1,\ldots n$ and $j=1,\ldots n-1$,
there exist differential operators $\cL_{ij}\in \bbK [\partial]$
and constants $a_i\in\bbK$ such that
\begin{equation*}
P_i(U)=a_i-\sum_{j=1}^{n-1} \cL_{ij} (u_j).
\end{equation*}
We define the differential polynomials
\[
T_i(X)=x_i-a_i, \,\,\, H_i(U)= \sum_{j=1}^{n-1} \cL_{ij} (u_j),
\,\,\mbox{and}\,\,\, F_i(X,U)= T_i(X)+H_{i}(U).
\]

Given $P\in \bbK \{X\cup U\}$ and $y\in X\cup U$, we denote by $\ord(P,y)$ the {\sf
order} of $P$ in the variable $y$. If $P$ does not have
a term in $y$ then we define $\ord(P,y)=-1$.

\begin{rem}\label{rem-ordFi}
To ensure that the number of parameters is $n-1$, we assume that
for each $j\in\{1,\ldots ,n-1\}$ there exists $i\in\{1,\ldots
,n\}$ such that the differential operator $\cL_{ij}\neq 0$. That
is, for each $j\in\{1,\ldots ,n-1\}$ there exists $i\in\{1,\ldots
,n\}$ such that $\ord (F_i, u_j)\geq 0$.
\end{rem}

\para

In this situation, the problem we deal with in this paper is:
given a system $\cP (X,U)$ of linear DPPEs, to compute an implicit
equation using differential resultants. In \cite{Gao},
algorithmic methods for solving this problem in a more general
case are presented in the language of characteristic sets. Our
candidate to be the implicit equation of $\cP (X,U)$ is the
differential resultant of $F_1,\ldots F_n$ that we describe in the
next section.

\section{Differential Resultants}\label{sec-Differential Resultants}

Let $\bbD$ be a differential integral domain, and let
$f_i\in\bbD\{U\}$ be an ordinary differential polynomial of order
$o_i$, $i=1,\ldots ,n$. A {\sf differential resultant} $\dres
(f_1,\ldots,f_n)$ of $n$ ordinary differential polynomials
$f_1,\ldots ,f_n$ in $n-1$ differential variables $u_1,\ldots
,u_{n-1}$ was introduced by Carra'Ferro in \cite{CFproc}. Such a
notion coincides with the Macaulay's algebraic resultant \cite{Mac} of the
differential polynomial set
\begin{displaymath}
\ps(f_1,\ldots ,f_n):=\{\partial^{N-o_i}f_i,\ldots
,\partial f_i, f_i \mid i=1,\ldots
,n,\,\,\,\mbox{where}\,\,\,N=\sum_{i=1}^n o_i\},
\end{displaymath}

\para

Now, let $h_i\in\bbD\{U\}$ be a ordinary differential
homogeneous polynomial of order $o_i$, $i=1,\ldots ,n$. We define the {\sf
differential homogenous resultant} $\dres^h (h_1,\ldots ,h_n)$ of
$n$ ordinary homogeneous differential polynomials $h_1,\ldots ,h_n$
in $n-1$ variables as the Macaulay's algebraic resultant of the
differential polynomial set
\begin{displaymath}
\ps^h(h_1,\ldots
,h_n):=\{\partial^{N-o_i-1}h_i,\ldots ,\partial h_i, h_i \mid
i=1,\ldots ,n,\,\,\,\mbox{where}\,\,\,N=\sum_{i=1}^n o_i\}.
\end{displaymath}

A differential homogeneous resultant was defined also by
Carra'Ferro in \cite{CF97} for $n=2$. In addition, when the
homogeneous polynomials have degree one and $n=2$ the differential
homogeneous resultant coincides with the differential resultant of
two differential operators defined by Berkovitch-Tsirulik in
\cite{BT} and studied also by Chardin in \cite{Ch}.

\para

We implemented in Maple the differential resultant matrices defined by
Carra'Ferro in \cite{CFproc}, our implementation is available at \cite{R}.
This is the tool used to perform our experiments and in particular
the computations in the examples of this paper.

Differential resultants are Macaulay's algebraic resultants
therefore with some previous computations, the implementation of
the Macaulay's algebraic resultant ( available at \cite{Min})
could be also used to compute differential resultants.

\begin{ex}
Let $\bbK=\bbC (t)$ and $\partial=\frac{\partial}{\partial t}$.
The differential resultant $\dres (f,g)$ of the differential
polynomials $f(u_1)=t - 4{u_{11}}^{2} - 4u_1^{2} - t{u_{11}}u_1 - 5{u_{11}}
- 4u_1$ and $g(u_1)=t - u_1 - 3{u_{11}}$ in $\bbK\{u_1\}$ ( $u_{1i}=\partial^i
u_1/\partial t^i$) is the Macaulay's algebraic resultant of the set
$\ps(f,g)=\{\partial f, f, \partial g, g\}$. By \cite{CFproc},
Definition 10, we can compute $\dres (f,g)$ as the quotient of two
determinants. The numerator is the determinant of a matrix of
order $20$, the number of monomials below, that we call $M(20)$ and the denominator is the
determinant of a submatrix of $M(20)$ that we call $A$.

The order of the differential polynomials in $\ps (f,g)$ is less
than or equal to two. The rows of the matrix $M(20)$ are the
coefficients of the following differential polynomials:

$\begin{array}{lcl} \mbox{ rows }1\ldots 4 &
\rightarrow & \{u_{11}\partial f, u_1\partial f, u_{12}\partial f,
\partial f\} \\
\mbox{ rows }5\ldots 8 & \rightarrow & \{{u_{11}}f,
u_1f, {u_{12}}f, f\}\\
\mbox{ rows }9\ldots 16 & \rightarrow &
\{{u_{11}}u_1\partial g, {u_{11}}{u_{12}}
\partial g, u_1{u_{12}}\partial g, {u_{12}}^{2}\partial g,\\
 & &\,{u_{11}}\partial g, u_1\partial g, {u_{12}}\partial g, \partial g\}\\
\mbox{ rows }17\ldots20 & \rightarrow & \{{u_{11}}u_1
g, {u_{11}}g, u_1g, g\}.
\end{array}
$

\para
The coefficients are written in decreasing order using first the
degree and then the lexicographic order with ${u_{12}}< u_1
<{u_{11}}$, that is, the columns of the matrix $M(20)$ are indexed
by the monomials in the list
\begin{align*}
&{u_{11}}^{3}, {u_{11}}^{2}u_1, {u_{11}}^{2}{u_{12}}, {u_{11}} u_1^{2},
{u_{11}}u_1{u_{12}}, {u_{11}}{u_{12}}^{2}, u_1^{3}, u_1^{2}{u_{12}},
u_1{u_{12}}^{2}, {u_{12}}^{3}, {u_{11}}^{2}
, {u_{11}}u_1,\\
&{u_{11}}{u_{12}}, u_1^{2}, u_1{u_{12}}, {u_{12}}^{2}, {u_{11}}, u_1, {u_{12}}, 1.
\end{align*}

The matrix $A$ is the submatrix of $M(20)$ obtained by removing
rows $1$, $2$, $5$, $6$, $10$, $11$, $12$ and columns $1$, $2$, $4$, $6$, $7$, $9$, $10$ of $M(20)$.
Finally, the differential resultant is the quotient $\dres
(f,g)=\det (M(20))/\det (A)$ where
\begin{align*}
&\det (M(20))=189050112\,t + 648075168\,t^{2} + 274613328\,t^{ 3}+ 1039857264\,t^{4}\\
&- 240663312\,t^{5} - 108661824\,t^{6} +1932336\,t^{7} + 3114288\,t^{8} - 174960\,t^{9},\\
&\det (A)= - 13608\,t^{2} - 64476\,t^{3} - 48600\,t^{4}+ 8748\,t^{5} - 3888\,t^{6}.
 \end{align*}
\end{ex}

\subsection{Differential resultant of $F_1,\ldots ,F_n$}\label{subsec-Differential resultants of F_1,...,F_n}

We consider now the polynomials $F_i$ and $H_i$, introduced in the
Section ~\ref{sec-Basic notions and notation}, and we set $\bbD=\bbK\{X\}$. In this section, we
give details on the computation of $\dres (F_1,\ldots ,F_n)$ and
$\dres^h (H_1,\ldots ,H_n)$, since they will be important tools
in this paper. We think of $F_1,\ldots ,F_n$ as polynomials in the
$n-1$ variables $u_1,\ldots ,u_{n-1}$ and coefficients in the
differential domain $\bbD$; recall that $F_1,\ldots ,F_n$ are  of
orders $o_i$ and degree one. We review below the computation of
$\dres (F_1,\ldots ,F_n)$ by means of determinants as in
\cite{CFproc}.

\para

We define rankings on the sets of variables $X$ and $U$ so that the
matrix we use to compute the differential resultant $\dres
(F_1,\ldots ,F_n)$ equals the one used in \cite{CFproc}.
\begin{itemize}
\item The order $x_n<\cdots <x_1$ induces a ranking on $X$ (i.e. an
order on $\{X\}$) as follows (see \cite{Kol}, page 75):
$x<\partial x$ and $x<x^\star\Rightarrow \partial^k x < \partial^{k^\star}
x^\star$, for all $x, x^\star\in X$, $k,k^\star\in\bbN_0$.

\item The order $u_1<\cdots <u_{n-1}$ induces an orderly ranking on $U$ as follows
(see \cite{Kol}, page 75): $u<\partial u$, $u<u^\star\Rightarrow
\partial u <\partial u^\star$ and $k<k^\star\Rightarrow\partial^k u <
\partial^{k^\star} u^\star$, for all $u, u^\star\in U$, $k,k^\star\in\bbN_0$. We
set $1<u_1$.
\end{itemize}

We call $\cR$ the ranking on $X\cup U$ that eliminates $X$ with
respect to $U$, that is $\partial^k x> \partial^{k^\star} u$, for all
$x\in X$, $u\in U$ and $k,k^\star\in\bbN_0$. Now, the set
$\ps=\ps(F_1,\ldots ,F_n)$ is ordered by $\cR$. Note that, because
of the particular structure of $F_i$, one has that:
\[ F_n<\partial   F_n <\cdots < \partial^{N-o_n} F_n< \cdots <F_2< \partial   F_2<\cdots< F_1<\cdots <\partial^{N-o_1} F_1.
\]
That is, $\ps$ is a chain (see \cite{Ritt}, page 3) of
differential polynomials $\{G_1,\ldots ,G_L\}$ with
$L=\sum_{i=1}^{n} N-o_i+1=(n-1)N+n$; recall that $N=\sum_{i=1}^n
o_i$.

\para

Then, let $M(L)$ be the $L\times L$ matrix whose $k$-th row
contains the coefficients of the $(L-k+1)$-th polynomial in $\ps$,
as a polynomial in $\bbD\{U\}$, and where the coefficients are
written in decreasing order with respect to the orderly ranking on
$U$. Hence, $M(L)$ is a matrix over $\bbK\{X\}$ that we call the
{\sf differential resultant matrix} of $F_1,\ldots ,F_n$. In this
situation:
\begin{displaymath}
\dres (F_1,\ldots ,F_n)=\det(M(L)).
\end{displaymath}

Analogously, we can use determinants to compute $\dres^h
(H_1,\ldots ,H_n)$; recall that the homogeneous differential
polynomials $H_i\in \bbK\{U\}$ are of orders $o_i$ and degree one.
Let $L^h=L-n$ and consider $\ps^h=\ps^h(H_1,\ldots ,H_n)$ as the
polynomial set obtained from $\ps$ by subtracting from the chosen
polynomials its monomial in $\bbD$ (i.e. $x_i-a_i$), therefore
maintaining in $\ps^h$ the ordering established in $\ps$. Let
$M(L^h)$ be the $L^h\times L^h$ matrix whose $(L^h-k+1)$-th row
contains the coefficients of the $k$-th polynomial in $\ps^h$, as
a polynomial in $\bbD\{U\}$, and where the coefficients are
written in decreasing order with respect to the orderly ranking on
$U$. Hence, $M(L^h)$ is a matrix over $\bbK$ that we call the {\sf
differential homogeneous resultant matrix} of $H_1,\ldots ,H_n$.
In this situation:
\begin{displaymath}
 \dres^h (H_1,\ldots ,H_n)=\det(M(L^h)).
\end{displaymath}

\begin{ex}
Let $\bbK=\bbQ(t)$, $\partial =\frac{\partial}{\partial t}$ and
consider the set of differential polynomials in $\bbK\{x_1,x_2,x_3\}\{u_1,u_2\}$,
\begin{align*}
F_1 (X,U)&= x_1-7t-u_1 - 3u_{11} + 3u_2 - tu_{21}\\
F_2 (X,U)&= x_2-u_1+u_{12} - 5u_{22}\\
F_3 (X,U)&= x_3-u_1+u_{12}- t^{2}u_{21}.
\end{align*}
Then the set $\ps(F_1,F_2,F_3)$ contains $L=13$ differential
polynomials and $\dres (F_1,F_2,F_3)=\det(M(L))$ where $M(L)$ is
the following $L\times L$ coefficient matrix of $\ps(F_1,F_2,F_3)$.

\[
 \left[
{\begin{array}{crcrcrcrcrrrc}
 - t & -3 & -1 & -1 & 0 & 0 & 0 & 0 & 0 & 0 & 0 & 0 & x_{14} \\
0 & 0 &  - t & -3 & 0 & -1 & 0 & 0 & 0 & 0 & 0 & 0 & x_{13} \\
0 & 0 & 0 & 0 &  - t & -3 & 1 & -1 & 0 & 0 & 0 & 0 & x_{12} \\
0 & 0 & 0 & 0 & 0 & 0 &  - t & -3 & 2 & -1 & 0 & 0 & x_{11}-7 \\
0 & 0 & 0 & 0 & 0 & 0 & 0 & 0 &  - t & -3 & 3 & -1 & x_1-7t
 \\
-5 & 1 & 0 & 0 & 0 & -1 & 0 & 0 & 0 & 0 & 0 & 0 & x_{23} \\
0 & 0 & -5 & 1 & 0 & 0 & 0 & -1 & 0 & 0 & 0 & 0 & x_{22} \\
0 & 0 & 0 & 0 & -5 & 1 & 0 & 0 & 0 & -1 & 0 & 0 & x_{21} \\
0 & 0 & 0 & 0 & 0 & 0 & -5 & 1 & 0 & 0 & 0 & -1 & x_2 \\
0 & 1 &  - t^{2} & 0 &  - 6\,t & -1 & -6 & 0 & 0 & 0 & 0 & 0 & x_{33} \\
0 & 0 & 0 & 1 &  - t^{2} & 0 &  - 4\,t & -1 & -2 & 0 & 0 & 0 & x_{32} \\
0 & 0 & 0 & 0 & 0 & 1 &  - t^{2} & 0 &  - 2\,t & -1 & 0 & 0 & x_{31} \\
0 & 0 & 0 & 0 & 0 & 0 & 0 & 1 &  - t^{2} & 0 & 0 & -1 & x_3
\end{array}}
 \right]
\]
Let $M(L^h)$ be the square submatrix of size $L^h=10$ of $M(L)$
obtained by removing columns $1$, $2$, $13$ and rows $1$, $6$, $10$. Then
$\dres^h (H_1,H_2,H_3)=\det(M(L^h))$.
\end{ex}

We show next some properties of $\dres (F_1,\ldots ,F_n)$ and
$\dres^h (H_1,\ldots ,H_n)$ that will be used later in the paper.

\begin{prop}\label{dResh1}
If  $\{H_1=0,\ldots ,H_n=0\}$ has a nonzero solution
then $\dres^h (H_1,\ldots ,H_n)=0$.
\end{prop}
\begin{proof}
Let $\bbF$ be a differential field extension of $\Bbb K$. Then
every nonzero solution of $\{H_1=0,\ldots ,H_n=0\}$ in
$\bbF^{n-1}$ is a nonzero solution of the system
$\{\partial^{N-o_i-1}H_i=0,\ldots ,\partial H_i=0, H_i=0 \mid
i=1,\ldots ,n\}$. If such a solution exists then the columns of
$M(L^h)$ are linearly dependent on $\bbF$.
\end{proof}

\begin{rem}\label{rem-HidResh}
For $n=2$, $\dres^h (H_1,H_2)=0$ if and only if $\{H_1=0,H_2=0\}$
has a nonzero solution in $\bbF$,  a differential field extension
of $\Bbb K$ (see \cite{BT}, Theorem 3.1). Unfortunately, for
$n>2$ the condition $\dres^h (H_1,\ldots ,H_n)=0$ is not
sufficient for the existence of nonzero solutions of the system
$\{H_1=0,\ldots ,H_n=0\}$. Let $n=3$ and
\[ H_1(U)=u_{11}+u_{21}, H_2(U)=u_1+u_2, H_3(U)=u_1+u_{11}+u_{21}. \]
The first two columns of $M(L^h)$ are equal and therefore $\dres^h (H_1,H_2,H_3)=0$.
The system $\{H_1=0,H_2=0,H_3=0\}$ has only the zero solution.
\end{rem}

We introduce some matrices that will be used in the proof of the
next result and also in later results in the paper. Let $S$ be the
$n\times (n-1)$ matrix whose entry $(i,j)$ is the
coefficient of $u_{jo_i}$ in $F_i$, $i\in\{1,\ldots ,n\}$,
$j\in\{1,\ldots ,n-1\}$. Let $S_i$ be the matrix obtained by
removing the $i$-th row of $S$.

\begin{rem}\label{S_dresh}
Note that the nonzero rows of the columns of $M(L)$ ( resp.
$M(L^h)$) corresponding to the coefficients of $u_{j N}$ (resp.
$u_{j N-1}$), $j=1,\ldots ,n-1$ are the rows of $S$.
\end{rem}

\begin{prop}\label{dRes0}
Let $M_{L-1}$ be the $L\times (L-1)$ principal
submatrix of $M(L)$. The following statements are equivalent:
\begin{enumerate}
\item $\dres (F_1,\ldots ,F_n)\neq 0$.

\item $\dres^h (H_1,\ldots ,H_n)\neq 0$.

\item \rank$(M_{L-1})=L-1$.

\end{enumerate}
\end{prop}
\begin{proof}
The matrix $M (L^h)$ is a square submatrix of $M_{L-1}$ so 3
implies 2. Given $i\in\{1,\ldots ,n\}$, we call $M_{L-1}^i$ the
submatrix of $M_{L-1}$ obtained by removing the row corresponding
to the coefficients of $\partial^{N-o_i} F_i$. Thus, if $\dres_h
(H_1,\ldots ,H_n)\neq 0$ then by Remark ~\ref{S_dresh} there
exists $k\in \{1,\ldots ,n-1\}$ such that $\det(S_k)\neq 0$.
Furthermore, there exists $a\in\bbN$ such that
\begin{displaymath}
\det(M_{L-1}^k)=(-1)^a \dres_h (H_1,\ldots ,H_n) \det(S_k)\neq 0,
\end{displaymath}
which proves that 2 implies 3.

Let $\bbK(X)$ be the quotient field of $\bbK\{X\}$. The
equivalence of 1 and 3 is obtained easily noting that the elements
of the last column of $M(L)$ are linearly independent on $\bbK(X)$
and so the last column of $M(L)$ is linearly independent on
$\bbK(X)$ from the first $L-1$ columns.
\end{proof}

\section{Implicitization of linear
DPPEs by \newline Carra'Ferros's differential resultant}\label{sec-Implicitization of linear DPPEs I}

Let $\cP(X,U)$, $F_i$, $H_i$ be as in Section \ref{sec-Basic
notions and notation}. Let $\bbD=\bbK\{X\}$, $\ps=\ps(F_1,\ldots,F_n)$ and let $\id$ be the implicit ideal of
$\cP(X,U)$. In this section, we use the implicitization results in terms of characteristic sets given in \cite{Gao} to obtain
implicitization results in terms of differential resultants. We prove that if $\dres (F_1,\ldots ,F_n)\neq 0$ then $\dres (F_1,\ldots ,F_n)(X)=0$ is
the implicit equation of $\cP(X,U)$.

\para

Recall that $\ps$ is a set of linear differential polynomials. Let $[\ps]$ be the
differential ideal generated by $\ps$, then it holds
$\id=[\ps]\cap \bbK\{X\}$ by \cite{Gao}, Lemma 3.2. Let $\cA$ be a
characteristic set of $[\ps]$ and $\cA_0=\cA\cap\bbK\{X\}$. By
\cite{Gao}, Theorem 3.1, the implicit ideal is
\[\id=[\ps]\cap \bbK\{X\}=[\cA_0].\]
Furthermore, if $\mid \cA_0\mid =1$
then $\cA_0=\{A(X)\}$ where $A$ is a characteristic polynomial of $\id$.
By \cite{CFproc}, Proposition 12, $\dres (F_1,\ldots
,F_n)\in \id$ and it is our candidate to be a characteristic polynomial of $\id$.

\para

Let $(\ps)$ be the ideal in
$\bbK[x_i,\ldots ,x_{i N-o_i} ,u_j,\ldots ,u_{j N}\mid i=1,\ldots
n, j=1,\ldots n-1 ]$ generated by $\ps$. To compute a
characteristic set of $\id$ we will use a Groebner basis
$\cG$ of $(\ps)$ with respect to the ranking $\cR^\star$ on $X\cup U$ that eliminates $U$ with
respect to $X$, that is $\partial^k x< \partial^{k^\star} u$, for all
$x\in X$, $u\in U$ and $k,k^\star\in\bbN_0$.

\begin{lem}\label{B0}
Let $\cG$ be the reduced Groebner basis of $(\ps)$ with respect to
the ranking $\cR^\star$.
\begin{enumerate}
\item $\cG=\{B_0,B_1,\ldots ,B_{L-1}\}$ where $B_0<B_1<\cdots
<B_{L-1}$ with respect to the ranking $\cR^\star$  and $\cG_0=\cG\cap
\bbK\{X\}$ is not empty.

\item Let $E(L)$ be the $L\times L$ matrix whose $k$-th row
contains the coefficients of $B_{L-k}$, $k=1,\ldots ,L$ as a polynomial in $\bbD\{U\}$, and where the coefficients are
written in decreasing order with respect to the orderly ranking on
$U$. Given the
differential resultant matrix $M(L)$ of $F_1,\ldots F_n$, then
$\det (M(L))=(-1)^a\det E (L)$ for some $a\in \bbN$.

\item The cardinality of $\cG_0$ is one if and only if $\dres^h (H_1,\ldots ,H_n)\neq 0$.
\end{enumerate}
\end{lem}
\begin{proof}
Let $M_{2L}$ be the $L\times (2L)$ matrix whose $k$-th row
contains the coefficients of the $k$-th polynomial in $\ps$, as a
polynomial in $\bbK\{X\cup U\}$, and where the coefficients are
written in decreasing order w.r.t. $\cR^\star$.

\[M_{2L}=\left[\begin{array}{cc}M_{L-1}&
\begin{array}{cccccccc}1& & & & & & &\partial^{N-o_1} a_1\\ &\ddots & & & & & &\vdots \\ & & 1& & & & & a_1 \\  & & &\ddots & & & &\vdots \\ & & & & 1 & & &\partial^{N-o_n} a_n\\ & & & & &\ddots & &\vdots \\  & & & & & & 1 &a_n \end{array}
\end{array}\right].\]

Observe that we can find a minimal Groebner basis of $(\ps)$ performing gaussian
elimination on the rows of $M_{2L}$, see \cite{Cox}, Section 7,
exercise 10. In particular, the polynomials corresponding to the
rows of the reduced echelon form $E_{2L}$ of $M_{2L}$ are the
elements of the reduced Groebner basis $\cG$ of $(\ps)$.

\begin{enumerate}
\item Observe that $M_{2L}$ has rank $L$. Therefore $\cG$ contains
$L$ elements $B_0<B_1<\cdots <B_{L-1}$. The submatrix formed by the first
$L-1$ columns of $M_{2L}$ is $M_{L-1}$ which has rank less than or
equal to $L-1$, therefore the set $\cG_0=\cG\cap \bbK\{X\}$ is not
empty.

\item The matrix $E(L)$ is obtained by performing on $M(L)$ the same
row operations as the operations performed on $M_{2L}$ to obtain $E_{2L}$.
Therefore $\det (M(L))=(-1)^a\det E (L)$ for some $a\in\bbN$.

\item The cardinality of $\cG_0$ is $1$ if and only if $\rank
(M_{L-1})=L-1$. Equivalently, $\dres^h (H_1,\ldots ,H_n)\neq 0$ by
Proposition ~\ref{dRes0}.
\end{enumerate}
\end{proof}

To compute a characteristic set of $[\ps]$ we apply the algorithm given in
\cite{BL}, Theorem 6, that we briefly include below for completion.
Given $P\in\bbK\{X\cup U\}$, the {\sf lead} of $P$ is the highest derivative
present in $P$ w.r.t. $\cR^\star$, we denote it by $\lead (P)$.
Given $P,Q\in \bbK\{X\cup U\}$ we denote by $\prem (P,Q)$ the
{\sf pseudo-remainder} of $P$ with respect to $Q$, \cite{Ritt}, page 7.
Given a chain $\cA=\{A_1,\ldots ,A_t\}$ of elements of $\bbK\{X\cup U\}$ then
$\prem (P,\cA)=\prem (\prem (P,A_t), \{A_1,\ldots
,A_{t-1}\})$ and $\prem (P,\emptyset)=P$.

\para

\begin{alg}\label{alg-characteristic set}
{\sf Given} the set of polynomials $\ps$ the next algorithm {\sf
returns} a characteristic set of $[\ps]$.
\begin{enumerate}
\item Compute the reduced Groebner basis $\cG$ of $(\ps)$ with
respect to $\cR^\star$.

\item Assume that the elements of $\cG$ are arranged in increasing
order $B_0<B_1<\cdots <B_{L-1}$ w.r.t. $\cR^\star$. Let $\cA=\{B_0\}$.
For $i$ from $1$ to $L-1$ do, if \lead$(B_i)\neq$\lead$(B_{i-1})$
then $\cA:=\cA \cup\{\mbox{\prem}(B_i,\cA)\}$.
\end{enumerate}
\end{alg}

\begin{thm}\label{implicit}
Given a system $\cP (X,U)$ of linear DPPEs with implicit ideal
$\id$. If $\dres^h (H_1,\ldots ,H_n)\neq 0$ then $\id$ has
dimension $n-1$ and
\begin{displaymath}
\dres (F_1,\ldots ,F_n)(X) =0
\end{displaymath}
is the implicit equation of $\cP (X,U)$.
\end{thm}
\begin{proof}
Let $\cG$ be the reduced Groebner basis of $(\ps)$ with respect to
the ranking $\cR^\star$ and let $B_0<B_1<\cdots <B_{L-1}$ be the
elements of $\cG$. By Lemma ~\ref{B0} we have
$\cG_0=\cG\cap\bbK\{X\}=\{B_0\}$ and so there exists a
characteristic set $\cA$ of $[\ps]$ such that $\cA_0 =\{B_0\}$.
Consequently, the dimension of $\id$ is $n-1$.

By Lemma ~\ref{B0} and the definition of the differential
resultant,
\[\dres (F_1,\ldots ,F_n)=\det(M(L))= (-1)^a \det
(E(L)),\] for some $a\in\bbN$, therefore $\dres (F_1,\ldots
,F_n)=c B_0$ with $c\in \bbK$. This proves that $\dres (F_1,\ldots
,F_n)$ is a characteristic polynomial of $\id$ and therefore the
implicit equation of $\cP (X,U)$ is $\dres (F_1,\ldots
,F_n)(X)=0$.
\end{proof}

\section{Linear Complete Differential Resultants}\label{sec-Linear gamma-Differential Resultants}

Let $\bbD$ be a differential integral domain, and let
$f_i\in\bbD\{U\}$ be a linear ordinary differential polynomial of
order $o_i$, $i=1,\ldots ,n$. For each $j\in \{1,\ldots n-1\}$ we
define the positive integer
\begin{displaymath}
\gamma_j (f_1,\ldots ,f_n)=\min\{o_i-\ord (f_i,u_j)\mid
i\in\{1,\ldots ,n\}\}.
\end{displaymath}
Observe that, $0\leq \gamma_j (f_1,\ldots ,f_n)\leq o_i$ for each $i\in\{1,\ldots
,n\}$. We also define
\begin{displaymath}
\gamma (f_1,\ldots ,f_n)=\sum_{j=1}^{n-1} \gamma_j
(f_1,\ldots ,f_n)
\end{displaymath}
and it is easily proved that $0\leq\gamma (f_1,\ldots ,f_n) \leq
N-o_i$, for all $i\in\{1,\ldots ,n\}$.

Observe that, when $\gamma (f_1,\ldots ,f_n) \neq 0$ then the set
$\ps (f_1,\ldots ,f_n)$ defined in Section ~\ref{sec-Differential Resultants}, is a set of $L$ differential polynomials
in $L-\gamma (f_1,\ldots ,f_n)-1$ differential variables. Then
$\dres (f_1,\ldots ,f_n)$ is the Macaulay's resultant of a set of
polynomials which are not complete in all its variables, thus
$\dres (f_1,\ldots ,f_n)=0$. If $\gamma (f_1,\ldots ,f_n)= 0$ we
will say that the set of differential polynomials $\{f_1,\ldots
,f_n\}$ is {\sf complete}. We will call $\gamma (f_1,\ldots ,f_n)$
the {\sf completeness index} of $\{f_1,\ldots ,f_n\}$.

We define next a differential resultant that generalizes
Carra'Ferro's differential resultant in the linear case. Observe
that the generalized differential resultant in \cite{CF05} deals
with a different aspect.

\begin{defi}
The {\sf linear complete differential resultant}
$\dcres (f_1,\ldots,f_n)$, of $n$ linear ordinary
differential polynomials $f_1,\ldots ,f_n$ in $n-1$ differential
variables $u_1,\ldots ,u_{n-1}$, is defined as the Macaulay's algebraic
resultant of the differential polynomial set
\begin{align*}
&\ps_{\gamma}(f_1,\ldots ,f_n):=\\
&\{\partial^{N-o_i-\gamma} f_i,\ldots ,\partial f_i, f_i \mid
i=1,\ldots ,n,\,\,\,N=\sum_{i=1}^n o_i,\,\,\, \gamma =\gamma
(f_1,\ldots ,f_n)\}.
\end{align*}
\end{defi}
The set $\ps_{\gamma}(f_1,\ldots ,f_n)$ contains
$L_{\gamma}=\sum_{i=1}^n (N-o_i-\gamma+1)$ polynomials in the following set
$\cV$ of $L_{\gamma}-1$ differential variables
\begin{equation}\label{eq-V}
\cV =\{u_j,\ldots ,u_{j N-\gamma_j-\gamma} \mid \gamma_j=\gamma_j
(f_1,\ldots ,f_n), j=1,\ldots ,n-1\}.
\end{equation}
Observe that $\sum_{j=1}^{n-1}(N-\gamma_j-\gamma+1)=(n-1)N-n\gamma
+n-1 =L_{\gamma}-1$.

\para
Let $h_i\in\bbD\{U\}$ be a linear ordinary
differential homogeneous polynomial of order $o_i$, $i=1,\ldots ,n$.
\begin{defi}
We define the {\sf linear complete differential homogenous resultant}
$\dcres^h (h_1,\ldots ,h_n)$, of $n$ linear ordinary
differential homogeneous polynomials $h_1,\ldots ,h_n$ in $n-1$
variables, as the Macaulay's algebraic resultant of the
differential polynomial set
\begin{align*}
&\ps_{\gamma}^h(h_1,\ldots ,h_n):=\\
&\{\partial^{N-o_i-\gamma -1}h_i,\ldots ,\partial h_i, h_i \mid
i=1,\ldots ,n,\,\,\,N=\sum_{i=1}^n o_i,\,\,\, \gamma =\gamma
(h_1,\ldots ,h_n)\}.
\end{align*}
\end{defi}
The set $\ps_{\gamma}^h(h_1,\ldots ,h_n)$ contains
$L_{\gamma}^h=\sum_{i=1}^n (N-o_i-\gamma)$ polynomials in the following set
$\cV^h$ of $L_{\gamma}^h-1$ differential variables
\begin{align*}
\cV^h =\{u_j,\ldots ,u_{j N-\gamma_j-\gamma -1}\mid
\gamma_j=\gamma_j (h_1,\ldots ,h_n), j=1,\ldots ,n-1\}.
\end{align*}

\para

Observe that for $\gamma(f_1,\ldots ,f_n)=0$ (resp. $\gamma (h_1,\ldots ,h_n)=0$)
it holds 
\begin{align*}
\dcres (f_1,\ldots ,f_n)&=\dres (f_1,\ldots ,f_n),\\ 
\dcres^h (h_1,\ldots ,h_n)&=\dres^h (h_1,\ldots ,h_n).
\end{align*}

We introduce next the matrices that will allow the use of
determinants to compute $\dcres (f_1,\ldots ,f_n)$ and $\dcres^h
(h_1,\ldots ,h_n)$. For $i=1,\ldots ,n$, $\gamma=\gamma
(f_1,\ldots ,f_n)$ (resp. $\gamma=\gamma (h_1,\ldots ,h_n)$) and
$k=1,\ldots ,N-o_i-\gamma$ (resp. $k=1,\ldots ,N-o_i-\gamma-1$)
define de positive integers,
\begin{align*}
l(i,k)&=(i-1)(N-\gamma)-\sum_{h=1}^{i-1}o_i+i+k,\\
l^h(i,k)&=(i-1)(N-\gamma-1)-\sum_{h=1}^{i-1}o_i+i+k.
\end{align*}
Then $l(i,k)\in \{1,\ldots ,L_{\gamma}\}$ and $l^h(i,k)\in
\{1,\ldots ,L_{\gamma}^h\}$, with the appropriate value of
$\gamma$ in each case.

Let $M(L_{\gamma})$ be the $L_{\gamma}\times L_{\gamma}$ (resp.
$L_{\gamma}^h\times L_{\gamma}^h$) matrix containing the
coefficients of $\partial^{N-o_i-\gamma -k}f_i$ ( resp.
$\partial^{N-o_i-\gamma -k-1} h_i$) as a polynomial in $\bbD[\cV]$ (resp. in $\bbD[\cV^h]$)
in the $l(i,k)$-th row  (resp. $l^h(i,k)$-th row), where the coefficients are written in
decreasing order with respect to the orderly ranking on $U$.
Hence, $M(L_{\gamma})$ (resp. $M(L_{\gamma}^h)$) is a matrix over
$\bbD$ that we call the {\sf complete differential (homogeneous)
resultant matrix} of $f_1,\ldots ,f_n$ (resp. $h_1,\ldots ,h_n$).
In this situation:
\begin{align*}
\dcres (f_1,\ldots ,f_n)&=\det(M(L_{\gamma})),\\
\dcres^h (h_1,\ldots ,h_n)&=\det(M(L_{\gamma}^h)).
\end{align*}

The next matrices will be used in the proofs of the results in
this section and also in later results in the paper. Let $S_{\gamma}$ be the $n\times (n-1)$ matrix whose
entry $(i,j)$ is the coefficient of $u_{j o_i-\gamma_j}$ in $f_i$,
$i\in\{1,\ldots ,n\}$, $j\in\{1,\ldots ,n-1\}$ For $i\in\{1,\ldots
n\}$ let $S_{\gamma i}$ be the $(n-1)\times (n-1)$ matrix obtained
removing the $i$-th row of $S_{\gamma}$.

\begin{rem}\label{rem-S_dreshgamma}
Let us suppose that $h_i$ is the homogeneous part of $f_i$,
$i=1,\ldots , n$, then $\gamma=\gamma (f_1,\ldots ,f_n)=\gamma
(h_1,\ldots ,h_n)$. Note that the nonzero rows of the columns of
$M(L_{\gamma })$ ( resp. $M(L_{\gamma}^h)$) corresponding to the
coefficients of $u_{j N-\gamma}$ (resp. $u_{j N-\gamma -1}$),
$j=1,\ldots ,n-1$ are the rows of $S_{\gamma}$.
\end{rem}

\begin{thm}\label{gdrthm1}
\begin{enumerate}
\item If the system $\{h_1=0,\ldots ,h_n=0\}$ has a nonzero
solution then $\dcres^h (h_1,\ldots ,h_n)=0$.

\item Let us suppose that $h_i$ is the homogeneous part of $f_i$,
$i=1,\ldots , n$. Let $M_{L_{\gamma}-1}$ be the $L_{\gamma}\times
(L_{\gamma}-1)$ principal submatrix of $M(L_{\gamma})$. Then,
\begin{enumerate}

\item $\dcres^h (h_1,\ldots ,h_n)\neq 0$ if and only if $\rank( M_{L_{\gamma}-1})=L_{\gamma}-1$,

\item if $\dcres (f_1,\ldots ,f_n)\neq 0$  then $\dcres^h (h_1,\ldots ,h_n)\neq 0$.
\end{enumerate}
\end{enumerate}
\end{thm}
\begin{proof}
\begin{enumerate}
\item The proof is analogous to the proof of Proposition
~\ref{dResh1}.

\item The matrix $M (L_{\gamma}^h)$ is a square submatrix of
$M_{L_{\gamma-1}}$. Given $i\in\{1,\ldots ,n\}$, we
call $M_{L_{\gamma}-1}^i$ the submatrix of $M_{L_{\gamma}-1}$
obtained by removing the row corresponding to the coefficients of
$\partial^{N-o_i-\gamma} f_i$. Thus, if $\dcres^h
(h_1,\ldots ,h_n)\neq 0$ then by Remark ~\ref{S_dresh} there
exists $k\in \{1,\ldots ,n-1\}$ such that $\det(S_{\gamma k})\neq
0$. Furthermore, there exists $a\in\bbN$ such that
\begin{displaymath}
\det(M_{L_{\gamma}-1}^k)=(-1)^a \dcres^h (h_1,\ldots ,h_n)
\det(S_{\gamma k})\neq 0,
\end{displaymath}
which proves the equivalence (a). Since $\dcres (f_1,\ldots
,f_n)\neq 0$ implies $\rank( M_{L_{\gamma}-1})=L_{\gamma}-1$,
statement (b) follows.
\end{enumerate}
\end{proof}

\begin{cor}
Let $F_i$ and $H_i$ be as in Section ~\ref{sec-Basic notions and notation} then
$\dcres (F_1,\ldots ,F_n)= 0$ if and only if $\dcres^h (H_1,\ldots ,H_n)= 0$.
\end{cor}
\begin{proof}
By Theorem ~\ref{gdrthm1}, 2(a), $\dcres (F_1,\ldots ,F_n)\neq 0$ if and
only if it holds $\rank( M_{L_{\gamma}-1})=L_{\gamma}-1$. The proof is
analogous to the proof of Proposition ~\ref{dRes0}.
\end{proof}

\begin{ex}
Let $\bbK=\bbQ(t)$, $\partial =\frac{\partial}{\partial t}$ and
consider the set of differential polynomials in
$\bbK\{x_1,x_2,x_3\}\{u_1,u_2\}$,
\begin{align*}
F_1 (X,U)&= x_1-t-u_1-u_2-2u_{22}\\
F_2 (X,U)&= x_2-t^2-2u_1-u_2-u_{22}\\
F_3 (X,U)&= x_3-5-u_1-3u_{11}-u_{21}-u_{23}.
\end{align*}
Then the set $\ps(F_1,F_2,F_3)$ contains $L=17$ differential
polynomials and $\dres (F_1,F_2,F_3)=\det(M(L))=0$ because the
columns of $M(L)$ corresponding to the coefficients of $u_{17}$
and $u_{16}$ are columns of zeros.

We have $\gamma_1=2$ and $\gamma_2=0$ so $\gamma=2$. Let us compute
$\dcres (F_1,F_2,F_3)$. The set
$\ps_{\gamma}(F_1,F_2,F_3)$ contains $L_{\gamma}=11$ differential
polynomials in the differential variables 
$\cV=\{u_{25}, u_{24},u_{23}, u_{13}, u_{22}, u_{12}, u_{21}, u_{11}, u_2, u_1\}$
written in decreasing order w.r.t. the ranking on $U$.
\[ M(L_{\gamma})=  \left[
{\begin{array}{rrrrrrrrrrc}
-2 & 0 & -1 & -1 & 0 & 0 & 0 & 0 & 0 & 0 & x_{13} \\
0 & -2 & 0 & 0 & -1 & -1 & 0 & 0 & 0 & 0 & x_{12} \\
0 & 0 & -2 & 0 & 0 & 0 & -1 & -1 & 0 & 0 & x_{11} - 1
 \\
0 & 0 & 0 & 0 & -2 & 0 & 0 & 0 & -1 & -1 & x_1 - t \\
-1 & 0 & -1 & -2 & 0 & 0 & 0 & 0 & 0 & 0 & x_{23} \\
0 & -1 & 0 & 0 & -1 & -2 & 0 & 0 & 0 & 0 & x_{22} - 2
 \\
0 & 0 & -1 & 0 & 0 & 0 & -1 & -2 & 0 & 0 & x_{21}- 2
\,t \\
0 & 0 & 0 & 0 & -1 & 0 & 0 & 0 & -1 & -2 & x_2 - t^{2} \\
-1 & 0 & -1 & -3 & 0 & -1 & 0 & 0 & 0 & 0 & x_{32} \\
0 & -1 & 0 & 0 & -1 & -3 & 0 & -1 & 0 & 0 & x_{31} \\
0 & 0 & -1 & 0 & 0 & 0 & -1 & -3 & 0 & -1 & x_3 - 5
\end{array}}
 \right]
\]

The matrix $M(L_{\gamma}^h)$ is obtained by removing rows $1$,
$5$, $9$ and columns $1$, $4$, $11$ of $M(L_{\gamma})$. Then,
\begin{align*}
\dcres
(F_1,F_2,F_3)=&4x_3-8-8x_{21}+12t+12x_{32}+4x_{13}+4x_{11}-20x_{23}\\
&-8x_{22}+4x_{12}
+4x_1-4x_2+4t^2\\
\dcres^h (H_1,H_2,H_3)=&-4.
\end{align*}
\end{ex}

\section{Implicitization of linear DPPEs by linear Complete Differential Resultants}\label{sec-Implicitization of linear DPPEs}

Let $\cP(X,U)$, $F_i$, $H_i$ be as in Section \ref{sec-Basic
notions and notation} and let $\id$ be the implicit ideal of
$\cP(X,U)$. In this section, we prove the main result of this
paper, namely if $\dcres^h (H_1,\ldots
,H_n)\neq 0$ then $\dcres (F_1,\ldots ,F_n)(X)$ is the
implicit equation of $\cP(X,U)$ and hence the results in Section ~\ref{sec-Implicitization of linear DPPEs I} are extended.

Let $\gamma=\gamma (F_1,\ldots ,F_n)=\gamma (H_1,\ldots ,H_n)$ and
$\cV$ as in ~\eqref{eq-V}. Let
$\ps_{\gamma}=\ps_{\gamma}(F_1,\ldots,F_n)$, $\bbD=\bbK\{X\}$ and
let $(\ps_{\gamma})$ be the ideal generated by $\ps_{\gamma}$ in $\bbK[x_i,\ldots ,x_{i
N-o_i-\gamma} ,u_j,\ldots ,u_{j N-\gamma}\mid i=1,\ldots n,
j=1,\ldots n-1 ]$. Let $\cR^\star$ be
the ranking on $X\cup U$ defined in Section
~\ref{sec-Implicitization of linear DPPEs I}.
\begin{lem}\label{B0gamma}
Let $\cG$ be the reduced Groebner basis of $(\ps_{\gamma})$ with respect to
the ranking $\cR^\star$.
\begin{enumerate}
\item $\cG=\{B_0,B_1,\ldots ,B_{L_{\gamma}-1}\}$ where
$B_0<B_1<\cdots <B_{L_{\gamma}-1}$ with respect to the ranking
$\cR^\star$ and $\cG_0=\cG\cap \bbK\{X\}$ is not empty.

\item Let $E(L_{\gamma})$ be the $L_{\gamma}\times L_{\gamma}$
matrix whose $k$-th row contains the coefficients of
$B_{L_{\gamma}-k}$, $k=1,\ldots ,L_{\gamma }$ as a polynomial in $\bbD[\cV]$,
and where the coefficients are written in decreasing order with respect to the orderly ranking on
$U$. Given the
differential resultant matrix $M(L_{\gamma})$ of $F_1,\ldots F_n$,
then $\det (M(L_{\gamma}))=(-1)^a\det E (L_{\gamma})$ for some
$a\in \bbN$.

\item The cardinality of $\cG_0$ is one if and only if $\dcres^h (H_1,\ldots ,H_n)\neq
0$.

\item If $\dcres^h (H_1,\ldots ,H_n)\neq 0$ then $\cA
=\{B_0,B_1,\ldots ,B_{n-1}\}$ is a characteristic set of
$[\ps_{\gamma}]$ and $\cA_0=\cA\cup \bbK\{X\}=\{B_0\}$.
\end{enumerate}
\end{lem}
\begin{proof}
The proof of 1,2 and 3 is analogous to the proof of Lemma
~\ref{B0}.

If $\dcres^h (H_1,\ldots ,H_n)\neq 0$ then $E(L_{\gamma})$
is an upper triangular matrix of rank $L_{\gamma}$ whose
$L_{\gamma}-1$ principal submatrix is the identity. Following
Algorithm ~\ref{alg-characteristic set} we obtain the characteristic set $\cA
=\{B_0,B_1,\ldots ,B_{n-1}\}$ of $[\ps_{\gamma}]$. Since
$\cG_0=\cG\cap\bbK\{X\}=\{B_0\}$ then $\cA_0 =\{B_0\}$.
\end{proof}

The next theorem gives an explicit formula of $\dcres
(F_1,\ldots ,F_n)$ in terms of $\dcres^h (H_1,\ldots
,H_n)$. Given $i\in\{1,\ldots ,n\}$, let $S_{{\gamma}i}$ be the
matrix defined in Section ~\ref{sec-Linear gamma-Differential Resultants}.
\begin{thm}\label{dresident}
\begin{enumerate}
\item There exists $P\in \id$ and $k\in\{1,\ldots ,n\}$ such that
\begin{enumerate}
\item $\ord(P,x_k)=N-o_k-\gamma$

\item $\ord(P,x_{i})\leq N-o_i-\gamma$, $i=1,\ldots,n$.
\end{enumerate}

\item For any differential polynomial $P$ as in statement 1 it holds
\begin{displaymath}
\dcres (F_1,\ldots ,F_n)= \frac{1}{\alpha} \det(S_{_{\gamma}k}) \dcres^h
(H_1,\ldots ,H_n) P(X).
\end{displaymath}
with $\alpha=(-1)^a\frac{\partial P}{\partial x_{k\;N-o_k-\gamma}}$,
$a\in\bbN$.
\end{enumerate}
\end{thm}
\begin{proof}
Given the reduced Groebner basis $\cG$ of $(\ps_{\gamma})$ w.r.t. $\cR^\star$
let us take a polynomial $B\in \cG_0=\cG\cap\bbK\{X\}$. Let
$\beta=\min\{N-o_i-\gamma-\ord (B,x_i)\mid i=1,\ldots ,n\} =N-o_k-\gamma-\ord
(B,x_k)$ for some $k\in\{1,\ldots ,n\}$. Then $P=\partial^{\beta}
B$ verifies the conditions in statement 1.

Since $P\in \id=[\ps_{\gamma}]\cap \bbK\{X\}$ there exist $\cF_1,\ldots
,\cF_n\in \bbK[\partial]$ such that $P(X)=\cF_1(F_1(X,U))+\cdots
+\cF_n(F_n(X,U))$. Then $P(X)=\cF_1(T_1(X))+\cdots +\cF_n(T_n(X))$
and $\cF_1(H_1(U))+\cdots +\cF_n(H_n(U))=0$. As a consequence, we
can perform row operations on $M(L_{\gamma})$ to obtain a matrix of the
kind
\begin{displaymath}
\left[\begin{array}{ccc}0\ldots 0& 0\ldots 0 &P(X)\\
S_{\gamma k}&\begin{array}{c}*\ldots *\\ \ddots \\ *\ldots * \end{array}&\begin{array}{c}*\\ \vdots \\ * \end{array}\\
\begin{array}{c}0\ldots 0\\ \ddots \\0\ldots 0 \end{array} & M(L_{\gamma}^h) & \begin{array}{c}*\\ \vdots \\ *\end{array}\end{array}\right].
\end{displaymath}
To be more precise, we reorder the rows of $M(L_{\gamma})$ so that the
coefficients of $\partial^{N-o_k-\gamma} F_k$ are in the first row, rows
$2$ to $n$ are the rows containing the entries of the matrix
$S_{\gamma k}$ and rows $n+1$ to $L_{\gamma}$ are the rows containing the
entries of $M(L_{\gamma}^h)$. Then multiply the first row of the
obtained matrix by $\frac{\partial P}{\partial x_{k\;N-o_k-\gamma}}$.
Finally, replace the first row by the coefficients of $P(X)$ as a
polynomial in $\bbD\{U\}$ written in decreasing order w.r.t. the
ranking on $U$, that is, all zeros and the last entry equal to
$P(X)$.

Let $\alpha=(-1)^a\frac{\partial P}{\partial x_{k\;N-o_k-\gamma}}$, for
some $a\in\bbN$ then
\begin{displaymath}
\alpha\det(M(L_{\gamma}))= \det(S_k)\dres_{\gamma}^h (H_1,\ldots ,H_n) P(X).
\end{displaymath}
\end{proof}

\begin{thm}\label{thm-gammaimp}
Given a system $\cP (X,U)$ of linear DPPEs with implicit ideal
$\id$. If $\dcres^h (H_1,\ldots ,H_n)\neq 0$ then $\id$ has
dimension $n-1$ and
\begin{displaymath}
\dcres (F_1,\ldots ,F_n)(X)=0
\end{displaymath}
is its implicit equation.
\end{thm}
\begin{proof}
By Lemma ~\ref{B0gamma}(4), a characteristic set of $\id$ is
$\cA_0=\{B_0\}$. Then the dimension of $\id$ is $n-1$. By Lemma ~\ref{B0gamma}(2), and the definition of the linear complete
differential resultant,
\[\dcres (F_1,\ldots ,F_n)=\det(M(L_{\gamma}))= (-1)^a \det
(E(L_{\gamma})),\] for some $a\in\bbN$. Therefore $\dcres
(F_1,\ldots ,F_n)=c B_0$ with $c\in \bbK$. This proves that
$\dcres (F_1,\ldots ,F_n)$ is a characteristic polynomial
of $\id$ and therefore the implicit equation of $\cP (X,U)$ is
$\dcres (F_1,\ldots ,F_n)(X)=0$.
\end{proof}

\section{Some results on properness}\label{sec-Some results on properness}
Let $\cP(X,U)$, $H_i$, $P_i$ and $\cL_{ij}$ be as in Section
\ref{sec-Basic notions and notation}. In this section, we give
some results related with the inversion problem. To be more
precise, we study conditions on the differential operators
$\cL_{ij}$ so that $\cP(X,U)$ is a set of proper DPPEs.

\para

We gather below some definitions that will be needed in this
section and that were used in \cite{Gao}, Section 6, to study the
inversion problem in terms of characteristic sets. The {\sf image}
of $\cP(X,U)$ is the set
\begin{displaymath}
\im=\{(\eta_1,\ldots ,\eta_n)\in\bbE^n\mid \exists (\tau_1,\ldots
,\tau_{n-1})\in \bbE^{n-1}\mbox{ with }\eta_i=P_i(\tau_1,\ldots
,\tau_{n-1})\}.
\end{displaymath}
The inversion problem says: given $(\eta_1,\ldots ,\eta_n)\in\im$,
find $(\tau_1,\ldots ,\tau_{n-1})\in \bbE^{n-1}$ such that
$\eta_i=P_i(\tau_1,\ldots ,\tau_{n-1})$. We call {\sf inversion
maps} for $\cP(X,U)$ to a set of funtions $g_1,\ldots ,g_{n-1}$ in
$\{X\}$ such that
\begin{displaymath}
u_j=g_j (x_1,\ldots ,x_n),\; j=1,\ldots ,n-1.
\end{displaymath}
A set $\cP(X,U)$ of DPPEs is {\sf proper} if for a generic zero
$(\Gamma_1,\ldots ,\Gamma_n)$ (and hence most of the points) of
the implicit variety $\zero(\id)$, there exists only one
$(\tau_1,\ldots ,\tau_{n-1})\in\bbE^{n-1}$ such that
$\Gamma_i=P_i(\tau_1,\ldots ,\tau_{n-1})$. By \cite{Gao}, Lemma
3.1, $(P_1(U),\ldots ,P_n(U))$ is a generic zero of the
implicit ideal $\id$.

\begin{prop}\label{prop-proper}
Let as suppose that $\dcres^h (H_1,\ldots ,H_n)\neq 0$,
then the next statements hold.
\begin{enumerate}
\item The set $\cP (X,U)$ of DPPEs is proper.

\item Furthermore, there exist a set of linear inversion maps for
$\cP (X,U)$,
\begin{displaymath}
U_1,\ldots ,U_{n-1}\in\bbK\{X\}.
\end{displaymath}
\end{enumerate}
\end{prop}
\begin{proof}
By Lemma ~\ref{B0gamma} then $E (L_{\gamma})$ is an upper
triangular matrix of rank $L_{\gamma}$ whose $L_{\gamma}-1$
principal submatrix is the identity. Then, $\cA =\{B_0,B_1,\ldots
,B_{n-1}\}$ is a characteristic set of $[\ps_{\gamma}]$ where
\begin{displaymath}
B_j(X,U)=u_j-U_j(X),\; j=1,\ldots ,n-1,
\end{displaymath}
for linear differential polynomials $U_j\in\bbK\{X\}$. By
\cite{Gao}, Theorem 6.1, the set $\cP(X,U)$ of DPPEs is proper and
$U_1,\ldots ,U_{n-1}$ is a set of inversion maps of $\cP(X,U)$.
\end{proof}

If $\bbK$ is not a field of constants with respect to $\partial$,
then $\bbK [\partial]$ is not commutative but
\[\partial k-k\partial=\partial (k)\]
for all $k\in \bbK$. The ring $\bbK [\partial]$ of differential
operators with coefficients in $\bbK$ is right euclidean (and also
left euclidean). Given $\cL,\cL'\in\bbK [\partial]$, by applying
the right division algorithm we obtain $q,r\in\bbK
[\partial]$, the right quotient and the right reminder of $\cL$
and $\cL'$ respectively, such that $\cL=q \cL'+r$ where $\deg
(r)<\deg(q)$.

If $\cL_j\in\bbK[\partial]$ is the {\sf greatest common right
divisor} of $\cL_{1j},\ldots ,\cL_{nj}$ then there exists
$\cL'_{ij}\in\bbK[\partial]$ such that $\cL_{ij}=\cL'_{ij}\cL_j$,
$i=1,\ldots ,n$. Then we write
\begin{displaymath}
\cL_j=\gcrd(\cL_{1j},\ldots ,\cL_{nj}),\; j=1,\ldots ,n-1.
\end{displaymath}
By Remark ~\ref{rem-ordFi}, then $\cL_j\neq 0$. If $\cL_j\in\bbK$,
then we say that $\cL_{1j},\ldots ,\cL_{nj}$ are {\sf coprime} and
we write
\begin{displaymath}
(\cL_{1j},\ldots ,\cL_{nj})=1.
\end{displaymath}

\begin{thm}\label{thm-proper}
A necessary condition for the set $\cP (X,U)$ of DPPEs to be
proper is
\begin{displaymath}
(\cL_{1j},\ldots,\cL_{nj})=1,\; j=1,\ldots ,n-1.
\end{displaymath}
\end{thm}
\begin{proof}
Let us suppose that there exists $k\in \{1,\ldots ,n-1\}$ such
that $\cL_k=\gcrd(\cL_{1k},\ldots ,\cL_{nk})$ is non constant.
Then there exists a nonzero element $\eta\in\bbK$ such that
$\cL_{ik}(\eta)=0$, $i=1,\ldots ,n$. Define the element
\begin{displaymath}
U+\eta=(u_1,\ldots, u_k+\eta,\ldots ,u_{n-1})\in \bbE^{n-1}.
\end{displaymath}
Recall that $P_i(U)=a_i-\sum_{j=1}^{n-1} \cL_{ij} (u_j)$, then
$(P_1(U+\eta),\ldots ,P_n(U+\eta))=(P_1(U),\ldots ,P_n(U))$. Thus
by definition, $\cP(X,U)$ is not proper.
\end{proof}

\begin{rem}\label{rem-LijProper}
For $n=2$ the condition in Theorem ~\ref{thm-proper} is also sufficient but for $n\geq 3$ this is not true.
Let $\bbK=\bbC (t)$ and $\partial=\frac{\partial}{\partial t}$. Let us consider de system of linear DPPEs
\begin{equation*}
\left\{\begin{array}{ccl}
x_1 & = & 2u_1+u_{11}+u_2+u_{22}\\
x_2 & = & u_1+u_{11}+u_{12}+u_2-u_{22}\\
x_3 & = & u_1+2u_{11}+u_2+u_{21}
\end{array}\right.,
\end{equation*}
where
\begin{equation*}
\begin{array}{ccc}
\cL_{11}=2+\partial & \cL_{21}= 1+\partial +\partial^2 & \cL_{31} = 1+2\partial\\
\cL_{12} = 1+\partial & \cL_{22} = 1+\partial^2 & \cL_{32} = 1+\partial.
\end{array}
\end{equation*}
We compute a characteristic set $\cA$ of $[\ps]$ with Algorithm
~\ref{alg-characteristic set}
\begin{align*}
\cA=\{ &x_{12}-2x_{22}-2x_{21}-x_2+x_{33}+x_{32}+x_{31}+x_3,\\
& u_2+3/2 u_1+x_{11}-1/2 x_1-2x_{21}-x_2+x_{32}+1/2 x_{31}+1/2x_3,\\
& u_{11}-u_1-x_{11}+x_1+2x_{21}-x_{32}-x_3
\}
\end{align*}
Then by \cite{Gao}, Theorem 6.1, the system is not proper but
$(\cL_{1j},\cL_{2j},\cL_{3j})=1$, $j=1,2$.
\end{rem}

We define a new system of DPPEs having the same implicit ideal
than $\cP(X,U)$
\begin{equation}\label{eq-P'}
\cP'(X,U) =\left\{\begin{array}{ccc}x_1&=& P'_1 (U)=a_1-H'_1(U)\\\vdots \\
x_n&=& P'_n (U)=a_n-H'_i(U)\end{array}\right.,
\end{equation}
where $H'_i(U)=\sum_{j=1}^{n-1} \cL'_{ij} (u_j)$ and $F'_i(U)=
T_i(X)+H'_i(U)$.

\begin{prop}\label{prop-id}
Let $\id'$ be the implicit ideal of the set $\cP'(X,U)$ of DPPEs.
Then $\id=\id'$.
\end{prop}
\begin{proof}
Given $f\in \id'$ then $f(P'_1(U),\ldots ,P'_n(U))=0$. In
particular, for $\eta=(\cL_1(u_1),\ldots ,\cL_{n-1}(u_{n-1}))$ then $0=f(P'_1(\eta ),\ldots ,P'_n(\eta ))=f(P_1(U),\ldots
,P_n(U))$. Therefore $\id'\subseteq\id$.

Now let us suppose that $f\in\id=[\ps]\cap\bbK\{X\}$. Then, there
exists $\cF_i\in\bbK [\partial]$, $i=1,\ldots ,n$ such that
$f(X)=\cF_1(F_1(X,U))+\cdots +\cF_n(F_n(X,U))$. Thus,
$\cF_1(H_1(U))+\cdots +\cF_n(H_n(U))=0$ and
$f(X)=\cF_1(T_1(X))+\cdots +\cF_n(T_n(X))$. As a consequence, for
each $j\in\{1,\ldots ,n-1\}$ we have $\cF_1(\cL_{1j}(u_j))+\cdots
+\cF_n(\cL_{nj}(u_j))=0$ . Thus $(\cF_1\cL'_{1j}+\cdots
+\cF_n\cL'_{nj})\cL_j=0$ and by Remark ~\ref{rem-ordFi} $\cL_j\neq
0$ so the differential operator $\cF_1\cL'_{1j}+\cdots
+\cF_n\cL'_{nj}=0$. We conclude that $\cF_1(H'_1(U))+\cdots
+\cF_n(H'_n(U))=0$ which implies $f(P'_1(U),\ldots ,P'_n(U))=0$
and proves $\id\subseteq\id'$.
\end{proof}

The following corollary follows directly from Theorem ~\ref{thm-gammaimp}
and Proposition ~\ref{prop-id}.

\begin{cor}\label{cor}
Given a system $\cP (X,U)$ of linear DPPEs with implicit ideal
$\id$. If $\dcres^h (H'_1,\ldots ,H'_n)\neq 0$ then $\id$ has
dimension $n-1$ and
\begin{displaymath}
\dcres (F'_1,\ldots ,F'_n)(X)=0
\end{displaymath}
is its implicit equation.
\end{cor}

\begin{ex}
The method to compute the implicit equation of the system $\cP (X,U)$  given by the equations
$x_1= u_1+u_{11}+u_2+u_{21}$; $x_2= t(u_{11}+u_{12})+u_{22}$; $x_3= u_1+u_{11}+u_{21}$ would be the following.
First, compute the system $\cP' (X,U)$ as in ~\eqref{eq-P'}, which is the system in Example ~\ref{ex-intro}. Then
the implicit equation of $\cP (X,U)$ is
$\dcres (F'_1,F'_2,F'_3)(X)=(t-1)x_{12}-tx_{31}-(t-1)x_{32}+x_2=0$.
\end{ex}

\section{Treatment of special cases}\label{sec-Treatment of special cases}

We give an explicit expression of the implicit equation of $\cP
(X,U)$ in terms of the differential operators defining the DPPEs
for $n=2$ and with some restrictions for $n=3$.

\subsection{Case n=2}\label{n2}

Given differential operators $\cL_1,\cL_2\in \bbK [\partial]$ we
consider the set of DPPEs
\begin{equation}
\cP_2(x_1,x_2,u) =\left\{\begin{array}{ccc}x_1 &= & a_1-\cL_1 (u)\\
x_2&= & a_2-\cL_2 (u)\end{array}\right.
\end{equation}
Let $H_1(u)=\cL_1(u)$, $H_2(u)=\cL_2(u)$ and
$F_1(x_1,x_2,u)=x_1-a_1+H_1(u)$, $F_2(x_1,x_2,u)=x_2-a_2+H_2(u)$.

\para

\begin{rem}\label{rem-chardin}
By \cite{Ch}, Theorem 2, it holds $\dres^h(H_1,H_2)\neq 0$ if and only if $(\cL_1,
\cL_2)=1$.
\end{rem}

If $\bbK$ is a field of constants with respect to $\partial$, that
is $\partial (k)=0$ for all $k\in \bbK$ ( for example
$\partial=\frac{\partial}{\partial t}$ and $\bbK=\bbC$) then
$\cL_1\cL_2-\cL_2\cL_1=0$, i.e. the differential operators
commute. If $\bbK$ is not a field of constants with respect to
$\partial$, then $\bbK [\partial]$ is not commutative. We need the
next commutativity result to prove the results in this section.
Most likely it was studied by Tsirulik in \cite{Tsi} but we have
not been able to find this paper. We give an algorithmic proof.

\begin{lem}\label{commutators}
Let $\cL_1,\cL_2\in\bbK[\partial]$ such that $(\cL_1, \cL_2)=1$.
Then there exist $\cD_1,\cD_2\in\bbK[\partial]$ with
$deg(\cD_i)\leq deg(\cL_i)-1$, $i=1,2$ such that
\begin{displaymath}
(\cL_2-\cD_2)\cL_1-(\cL_1-\cD_1)\cL_2=0.
\end{displaymath}
\end{lem}
\begin{proof}
If $\cL_1$ and $\cL_2$ commute then $\cD_i=0$, $i=1,2$. Let us
suppose that $\cL_1\cL_2-\cL_2\cL_1\neq 0$. Let
$\cL_1=\sum_{i=0}^{o_1}\phi_i\partial^i$ and
$\cL_2=\sum_{j=0}^{o_2}\psi_j\partial^j$ with $\phi_i,\psi_j\in\bbK$. 
By equation (1.2) in \cite{BT}, we have
\begin{displaymath}
\cL_1\cL_2-\cL_2\cL_1=\sum_{k=0}^{o_1+o_2} (c_k-\bar{c}_k)
\partial^k\mbox{ where }
\end{displaymath}
\begin{displaymath}
c_k=\sum_{s=max(0,k-o_2)}^{min(o_1,k)} \sum_{i=s}^{o_1} \phi_i
\partial^{(i-s)} (\psi_{k-s}),\indent
\bar{c}_k=\sum_{s=max(0,k-o_1)}^{min(o_2,k)} \sum_{j=s}^{o_2} \psi_j
\partial^{(j-s)} (\phi_{k-s}).
\end{displaymath}
Then $c_{o_1+o_2}=\bar{c}_{o_1+o_2}=\phi_{o_1}\psi_{o_2}$. Therefore the
degree of the operator $\cL_1\cL_2-\cL_2\cL_1$ is less than or
equal to $o_1+o_2-1$.

We want to find $\cD_1=\sum_{i=0}^{o_1-1}\alpha_i\partial^i$ and
$\cD_2=\sum_{j=0}^{o_2-1}\beta_j\partial^j$ with
$\alpha_i,\beta_j\in \bbK$ such that
$\cD_1\cL_2-\cD_2\cL_1=\cL_1\cL_2-\cL_2\cL_1$. We have
\begin{displaymath}
\cD_1\cL_2-\cD_2\cL_1=\sum_{k=0}^{o_1+o_2-1}
(\gamma_k-\bar{\gamma}_k)
\partial^k\mbox{ where }
\end{displaymath}
\begin{displaymath}
\gamma_k=\sum_{s=max(0,k-o_2)}^{min(o_1-1,k)} \sum_{i=s}^{o_1-1}
\alpha_i \partial^{(i-s)} (\psi_{k-s}),\indent
\bar{\gamma}_k=\sum_{s=max(0,k-o_1)}^{min(o_2-1,k)}
\sum_{j=s}^{o_2-1} \beta_j \partial^{(j-s)} (\phi_{k-s}).
\end{displaymath}

Let us consider the system of $o_1+o_2$ equations
\begin{equation}\label{system}
\gamma_k-\bar{\gamma}_k=c_k-\bar{c}_k, \indent k=0,\ldots
,o_1+o_2-1
\end{equation}
in the $N=o_1+o_2$ unknowns $\alpha_i$, $i=0,\ldots ,o_1$ and
$\beta_j$, $j=0,\ldots ,o_2$. The coefficient matrix of the system
~\eqref{system} is the differential homogeneous resultant matrix
$M(L^h)$ of $H_1,H_2$, with $L^h=N$. Thus the system has a unique
solution since $\det(M(L^h))=\dres^h(H_1,H_2)\neq 0$, by Remmark
~\ref{rem-chardin}.
\end{proof}

\begin{thm}
Given a system $\cP_2 (x_1,x_2,u)$ with implicit ideal $\id$.
\begin{enumerate}
\item The following statements are equivalent.
\begin{enumerate}
\item $\cP_2 (x_1,x_2,u)$ is proper.

\item $(\cL_1,\cL_2)=1$.

\item $\dres^h (H_1,H_2)\neq 0$.
\end{enumerate}

\item If $(\cL_1,\cL_2)=1$ then the dimension of $\id$ is $1$ and
$\dres (F_1,F_2)(x_1,x_2)=0$ is the implicit equation of $\cP_2
(x_1,x_2,u)$. Furthermore, there exist
$\cD_1,\cD_2\in\bbK[\partial]$ such that
\begin{align*}
&\dres (F_1,F_2)(x_1,x_2)=\\
&(-1)^a\dres^h(H_1,H_2)[(\cL_2-\cD_2)(x_1-a_1)-(\cL_1-\cD_1)(x_2-a_2)],
\end{align*}
for some $a\in\bbN$.
\end{enumerate}
\end{thm}
\begin{proof}
\begin{enumerate}
\item It follows from Proposition ~\ref{prop-proper}, Theorem ~\ref{thm-proper} and Remark ~\ref{rem-chardin}.

\item Observe that for $n=2$, $\gamma(F_1,F_2)=\gamma(H_1,H_2)=0$.
By statement 1 and Theorem ~\ref{thm-gammaimp}, $\dres
(F_1,F_2)(x_1,x_2)=0$ is the implicit equation of $\cP_2
(x_1,x_2,u)$. By Lemma ~\ref{commutators} there exist $\cD_1,\cD_2\in \bbK
[\partial]$ such that
$(\cL_2-\cD_2)\cL_1(u)-(\cL_1-\cD_1)\cL_2(u)=0$. Let
\begin{displaymath}
P(x_1,x_2)=(\cL_2-\cD_2)(x_1-a_1)-(\cL_1-\cD_1)(x_2-a_2).
\end{displaymath}
Since $\dres^h (H_1,H_2)\neq 0$ then there exists $k\in\{1,2\}$
such that $\det (S_k)=\partial P/ \partial x_{k N-o_k}\neq 0$.
The result follows from Theorem ~\ref{dresident}.
\end{enumerate}
\end{proof}

\begin{alg}
{\sf Given} the system $\cP_2 (x_1,x_2,u)$ of DPPEs this algorithm
{\sf returns} its implicit equation $A(X)$.
\begin{enumerate}
\item $\cL:=\gcrd (\cL_1 ,\cL_2)$.

\item If $\cL\notin\bbK$ then compute $\cL'_i$ such that
$\cL_i=\cL'_i\cL$ and set $\cL_i:=\cL'_i$, $i=1,2$.

\item Compute $\cD_1$ and $\cD_2$.

\item $A(X):=(\cL_2-\cD_2)(x_1-a_1)-(\cL_1-\cD_1)(x_2-a_2)$.
\end{enumerate}
\end{alg}

The Maple packages OreTools, DEtools and Ore-algebra
can be used to compute $\gcrd(\cL_1,\cL_2)$, \cite{ALL}.
Also, we can compute $\cD_1$ and $\cD_2$ following the
algorithm in the proof of Lemma ~\ref{commutators}.

\subsection{Case n=3}
We restrict to the case $\bbK=\bbC$ and consider the system of
DPPEs
\begin{equation}
\cP_3(X,U) =\left\{\begin{array}{ccc}
x_1 &= & a_1-\cL_{11}(u_1)-\cL_{12}(u_2)\\
x_2 &= & a_2-\cL_{21}(u_1)-\cL_{22}(u_2)\\
x_3 &= & a_3-\cL_{31}(u_1)-\cL_{32}(u_2)
\end{array}\right.
\end{equation}
where $X=\{x_1,x_2,x_3\}$, $U=\{u_1,u_2\}$ and $\cL_{ij}\in\bbC
[\partial ]$. These differential operators commute, therefore the
polynomial
\begin{align*}
P(X)=&\cL_{21}\cL_{32}(x_1-a_1)-\cL_{22}\cL_{31}(x_1-a_1)-\cL_{11}\cL_{32}(x_2-a_2)+\\
&+\cL_{12}\cL_{31}(x_2-a_2)+\cL_{11}\cL_{22}(x_3-a_3)-\cL_{12}\cL_{21}(x_3-a_3)
\end{align*}
belongs to the implicit ideal of $\cP_3(X,U)$.

For $i=1,2,3$, we have $F_i(X,U)=a_i+H_i(U)$ and
$H_i(U)=\cL_{i1}(u_1)+\cL_{i2}(u_2)$ of order $o_i$. Let
$\gamma=\gamma (F_1,F_2,F_3)$ and let $S_{\gamma i}$ be the matrix
defined in Section ~\ref{sec-Linear gamma-Differential Resultants}. It is easily obtained that
\begin{displaymath}
\det (S_{\gamma i})=\frac{\partial P}{\partial x_{i
N-o_i-\gamma}}.
\end{displaymath}
If $\det (S_{\gamma i})\neq 0$ for some $i\in\{1,2,3\}$ then it follows
from Theorem ~\ref{dresident} that
\begin{displaymath}
\dcres (F_1,F_2,F_3)(X)=(-1)^a\dcres^h
(H_1,H_2,H_3)P(X),
\end{displaymath}
$a\in\bbN$.

\end{document}